\documentclass{amsart}

\usepackage{amssymb,amsmath}
\usepackage[alphabetic,abbrev,lite]{amsrefs} % for bibliography
\usepackage{latexsym} 
\usepackage{amsfonts}
\usepackage{amsthm}
\usepackage{mathrsfs}
\usepackage{verbatim} % needed for \begin{comment} \end{comment}
\usepackage{xspace}
\usepackage[pdftex,colorlinks=true, pdfstartview=FitV, linkcolor=magenta, citecolor=blue, urlcolor=blue]{hyperref}
\usepackage{tikz}
\usepackage[colorinlistoftodos, textwidth=3cm, shadow]{todonotes}
\newcommand{\st}{~|~}

\newcommand{\Zz}{\mathbb{Z}}

\newcommand{\ra}{\rightarrow}
\newcommand{\ds}{\displaystyle}

\newcommand{\vl}{\;\vert\;}

\newcommand{\Z}{\mathbb{Z}}

\newcommand{\x}{\mathbf{x}}

\newcommand{\ls}{\leqslant}
\newcommand{\gs}{\geqslant}

\newcommand{\lst}[2]{#1_1,\dots,#1_{#2}}
\newcommand{\kk}{\Bbbk}
\newcommand{\bb}[1]{{\bf #1 }}

\newcommand{\dfn}[1]{\textsf{#1}\index{#1}}

\newcommand{\B}[1]{B{(#1)}} % Fibanocci bound
\newcommand{\lex}{\ensuremath{\text{\textsc{lex}}}\xspace}
\newcommand{\lamb}[2]{\lambda_{#1}(#2)}
\newcommand{\lam}[1]{\lambda_{#1}}

\theoremstyle{plain} % default.  Heading is bold, text italic.
\newtheorem{thm}{Theorem}
\numberwithin{thm}{section}
\newtheorem{prop}[thm]{Proposition}
\newtheorem{cor}[thm]{Corollary}
\newtheorem{lem}[thm]{Lemma}

\theoremstyle{definition}

\newtheorem{example}[thm]{Example} 
\newtheorem{defn}[thm]{Definition}

\theoremstyle{remark}  % Heading is italic, text is roman
\newtheorem{rem}[thm]{Remark}

\title{Sequences defined by $h$-vectors}

\author[T. Enkosky and B. Stone]{Thomas Enkosky and Branden Stone}
\address{Thomas Enkosky, U.S. Coast Guard Academy, Department of Mathematics, 15 Mohegan Ave, New London, CT, 06320 }
\email{Thomas.A.Enkosky@uscga.edu}
\address{Branden Stone, Mathematics Program, Bard College, P.O. Box 5000, Annandale-on-Hudson, NY 12504}
\email{bstone@bard.edu}


\begin{document}

\maketitle 

\begin{abstract}
	In this paper we consider the sequence whose $n^{th}$ term is the number of $h$-vectors of length $n$.  We show that the $n^{th}$ term of this sequence is bounded above by the $n^{th}$  Fibonacci number and bounded below by the number of integer partitions of $n$ into distinct parts. Further we show embedded sequences that directly relate to integer partitions. 
\end{abstract}

\section{Introduction}

Hilbert functions of graded rings have been well studied throughout the years and are known to relate to many different invariants such as dimension, multiplicity, and Betti numbers \cite[Chapter 4]{BH}. In 1927, Macaulay showed that for every graded ideal there exists a \lex-segment with the same Hilbert function \cite{macaulay27}. Since then a wide range of research has accumulated generalizing this result \cite{CL, FR, MM,manoj}.  These functions have a variety of uses in both algebra and combinatorics and are the subject of active research \cite{PS}. In particular, it is helpful giving necessary conditions for a ring to have the weak Lefschetz property \cite{harima03}.

In \cite{linusson99}, Linusson counted sequences and vectors associated to Hilbert functions. In particular, recursion formulas were given for the number of $M$-sequences (i.e. $f$-vectors for multicomplexes) in terms of the number of variables and a maximum degree.  When the number of variables was restricted to 3, it was shown that the Bell numbers counted the number of $M$-sequences. In recent work of Whieldon \cite{gwyn}, given certain classes of monomial ideals, the sequence of Betti numbers satisfies nice recursion formulas. In particular, the Betti numbers of the resolution of $\kk$ over $S = \kk[x,y]/(x^2,xy)$ are given by the $i^{th}$ Fibonacci number!  The goal of this paper is to find recursion formulas related to Hilbert functions.  We are mainly concerned with the sequence defined by the number of $h$-vectors of length $n$.   We show that this sequence is bounded above (term-wise) by the sequence of Fibonacci numbers and below by the number of integer partitions of $n$ into at least 2 distinct parts. As such, the sequence has exponential growth.

The rest of this section gives the necessary background and notation.  In Section~\ref{sec:fib-bound} we determine an upper bound for our sequence to be the Fibonacci numbers. The lower bound can be found in Section \ref{sec:int-part} as well as a one-to-one correspondence between integer partitions and \lex ideals in two variables. The rest of the paper generalizes these concepts.

\subsection{Basic Setup}

	We first give some necessary background on Hilbert functions and $h$-vectors.  Let $R = \kk[\lst x n]$ be a polynomial ring over a field $\kk$ with the standard grading.  In particular, $\deg x_i = 1$ for $1\ls i\ls n$.  If $I$ is a graded ideal, the quotient ring $R/I$ is also graded and we denote by $(R/I)_t$ the $\kk$-vector space of all degree $t$ homogeneous elements of $R/I$.  The \dfn{Hilbert function} $H_{R/I}~:~\Zz_{\gs 0} \ra \Zz_{\gs 0} $ is defined to be the $\kk$-vector space dimension  of each graded component, i.e. $\ds H_{R/I}(t):= \dim_k (R/I)_t$. 
	
	If the Krull dimension of the graded quotient ring is zero, there exists an $s\gs 0$ such that  $H_{R/I}(s) \not= 0$ but $H_{R/I}(t) = 0$ for all $t>s$. In this case, the \dfn{$h$-vector} of $R/I$ is defined as
	\[
		\bb h(R/I) = (H_{R/I}(0), H_{R/I}(1), H_{R/I}(2),\ldots, H_{R/I}(s)).
	\]
Thus the $h$-vector of $R/I$ has finitely many non-zero entries. The \dfn{length}  of $R/I$ is the $k$-vector space dimension of $R/I$, denoted $\lambda(R/I)$.  In particular, $\lambda(R/I)=\sum_{i=0}^s H_{R/I}(i)$. Throughout this paper, we will also refer to $\lambda(R/I)$ as the \dfn{length of $\bb h (R/I)$}.

	In \cite[Chapter 4]{BH} a numerical constraint is given on the possible integer vectors that can be $h$-vectors.  Given $d\in\Z_{\gs 0}$, each $a\in\Z_{\gs 0}$ has a unique representation as a sum of binomial coefficients
	\begin{equation}\label{eqn:binom-rep}
	  a=\binom{b_d}{d}+\binom{b_{d-1}}{d-1}+\cdots+\binom{b_j}{j},
	\end{equation}
where $\ds b_d>b_{d-1}>\cdots>b_j\gs j\gs 1$.  Further, define
	\begin{equation}\label{eqn:macaulay-coeff}
	  a^{\langle d \rangle}=\binom{b_d+1}{d+1}+\binom{b_{d-1}+1}{d}+\cdots+\binom{b_j+1}{j+1},
	\end{equation}
where $0^{\langle d \rangle} = 0$. For a map $h:\Z_{\gs 0} \to \Z_{\gs 0}$, Macaulay's Theorem \cite[Theorem 4.2.10]{BH} says the following conditions are equivalent: 
\begin{enumerate}
	\item[(A)] there exists a graded ideal $I$ in $R$ such that $H_{R/I}(t) = h(t)$ for all $t \gs 0$;
	\item[(B)] there exists a monomial ideal $I$ in $R$ such that $H_{R/I}(t) = h(t)$ for all $t \gs 0$;
	\item[(C)] one has $h(0) = 1$, and $h(t+1) \ls h(t)^{\langle t \rangle}$ for all $t \gs 1$.
\end{enumerate}

	Throughout this paper, for an arbitrary set $\Lambda$, we denote $|\Lambda|$ as the cardinality of $\Lambda$.

\section{Fibonacci Bound}\label{sec:fib-bound}

The main study of this paper is the sequence $\ds\left\{\ell(n)\right\}_{n\gs 1}$ defined by the number of $h$-vectors of length $n$.   In particular, for $n\gs 1$ we define 
\[
    L(n)=\left\{h=(h_0,h_1,\dots) ~|~ h\text{ is an $h$-vector and } \sum_i  h_i=n\right\}.
\]
and set $\ds\ell(n)=|L(n)|$ for $n\gs 1$.

Using condition (C) in Macaulay's theorem above, we are able to construct all possible $h$-vectors of a given length.  In Figure \ref{fig:hvec}, we find the $h$-vectors of length at most 7 and that the first few terms of the sequence $\ds\left\{\ell(n)\right\}_{n\gs 1}$ are: 1, 1, 2, 3, 5, 8, 12.

\begin{figure}[h!] 
\begin{tabular}{llllllll}
	$\lambda = $ & 1 & 2 & 3 & 4 & 5 & 6 & 7 \\
	\hline
	\hline
	    & 1 & 11 & 111 & 1111 & 11111 & 111111 & 1111111 \\
	  	&   &	 & 12  & 121  & 1211  & 12111  & 121111 \\
	   &	&	 &     & 13   & 122   & 1221   & 12211 \\
	  	& 	& 	 &     &      & 131   & 123    & 1231 \\
	  	& 	& 	 &     &      & 14    & 1311   & 1222 \\
	  	& 	& 	 &     &      &		  & 132    & 13111 \\
	  	& 	& 	 &     &      &		  & 141    & 1321 \\
	  	& 	& 	 &     &      &		  & 15    & 133  \\
		& 	& 	 &     &      &		  &		   & 1411 \\
		& 	& 	 &     &      &		  &		   & 142 \\
		& 	& 	 &     &      &		  &		   & 151 \\
		& 	& 	 &     &      &		  &		   & 16 \\
		\hline		
	Total: & 1 & 1 & 2  & 3  & 5  & 8  & 12									
\end{tabular}
\caption{The $h$-vectors of length at most 7. We write $t_0t_1t_2\cdots t_s$ for the $h$-vector $(t_0, \lst t s)$. E.g. 1221 is the $h$-vector $(1,2,2,1)$. 
}\label{fig:hvec}
\end{figure}

After seeing the first few terms of this sequence, a natural question to ask is whether or not it is related to the Fibonacci sequence.  
In Theorem \ref{thm:upper-bound} we show $\ds\left\{\ell(n)\right\}_{n\gs 1}$ is bounded above by the Fibonacci sequence.   
To do this, we need to define the following family of integer vectors.

\begin{defn}\label{def:Bn}
	For $n\gs 1$, the set of integer vectors $\ds \B n$ is defined recursively as follows:
	\begin{enumerate}
	\item $\ds B(1)=\{(1)\}$;
	\item $\ds B(2)=\{(1,1)\}$;
	\item For $n\gs 3$ define $\ds B(n):=C(n)\cup D(n)$ where
\begin{align*}
C(n) &:=\left\{(1,\lst t s,1)~|~(1,\lst t s)\in B(n-1)\right\},\\
D(n) &:=\left\{(1,\lst t s+1)~|~(1,\lst t s)\in B(n-1),~\text{with } t_{s-1}>1\text{ or } s=1 \right\}.
\end{align*}
	 	
	\end{enumerate}
\end{defn}

\begin{rem}\label{rem:Bn}
	It is worth noticing that the sets $C(n)$ and $D(n)$ of Definition \ref{def:Bn} form a set partition of $B(n)$.
\end{rem}

The first few sets $B(n)$ are
\begin{align*}
B(1)&=\{(1)\};\\
B(2)&=\{(1,1)\};\\
B(3)&=\{(1,1,1),(1,2)\};\\
B(4)& = \{(1,1,1,1), (1,2,1), (1,3) \};\\
B(5)&=\{(1,1,1,1,1),(1,2,1,1),(1,3,1),(1,2,2),(1,4)\}.
\end{align*}

\begin{lem}\label{lem:fib-Bn}
  The cardinality of $B(n)$ is the $n^{\text{th}}$ Fibonacci number $F_n$.
\end{lem}

\begin{proof} 
  For notational convenience, we let $b_n=|B(n)|$ and observe that $b_1=b_2=1$ and $b_3=2$.  We need to show that for $n\gs 4$ this sequence satisfies the Fibonacci recurrence
  \begin{align*}
    b_{n} &= b_{n-1}+b_{n-2}\\
      &= b_{n-1}+b_{n-1}-b_{n-3}.
  \end{align*}

	By Remark \ref{rem:Bn}, we know $\ds b_n=|C(n)|+|D(n)|$.  Since $\ds |C(n)|=b_{n-1}$, we need to show that $|D(n)|=b_{n-1}-b_{n-3}$.  Partition the set $B(n-1)=E(n-1)\cup F(n-1)$ so that $E(n-1)$ is the set of vectors $v \in B(n-1)$ such that the last two entires of $v$ equal 1, and $F(n-1)$ consists of the remaining vectors.   A vector is in $E(n-1)$ if and only if it came from $B(n-3)$ by adjoining  1 twice in accordance to the recursion in Definition \ref{def:Bn}. In particular, we have that
	\[
		 |E(n-1)|=|C(n-2)|=b_{n-3}. 
	\]

We claim that $|D(n)| = |F(n-1)|$, and hence $|D(n)| = b_{n-1}- b_{n-3}$.  To see this, notice that $F(n-1)\subseteq B(n-1)$ consists of all vectors whose second to last term is greater than 1.  Hence we can increase the last term of any vector in $F(n-1)$ by 1, and the resulting vector is in $B(n)$.  As such, all the vectors in $D(n)$ will come from vectors in $F(n-1)$, hence $|D(n)|=|F(n-1)|$.
\end{proof}

\begin{thm}\label{thm:upper-bound} 
For all $n\gs 1$, $\ds L(n)\subseteq B(n)$.  In particular, the sequence $\ell(n)$ is bounded above by the Fibonacci sequence.
\end{thm}
\begin{proof}
	Notice that the set $B(n)$ consists of all integer vectors $(1,\lst t s)$ with $1+t_1+t_2+\cdots+t_s=n$ and the property that if $t_i=1$, then $t_j=1$ for all $j\gs i$. Let $\bb h = (1,\lst h s)\in L(n)$ be an $h$-vector of length $n$.  Using Macaulay's Theorem condition (C) it is not hard to see that if $h_i=1$ for some $i \gs 1$, then $h_j=1$ for all $j\gs i$.  Thus $L(n)\subseteq B(n)$.
\end{proof}

\begin{rem}
For $n\gs 7$, there are elements of $\B n$ that are not $h$-vectors.  See $(1,2,4)$ for an example.  Further analysis shows the first 20 terms of $\ell(n)$ shows this upper bound is not tight:
\[
 \{\ell(n)\}_{n\gs 1} = 1, 1, 2, 3, 5, 8, 12, 18, 27, 40, 57, 82, 116, 163, 227, 313, 428, 583, \ldots .
\]
\end{rem}

\section{Integer Partitions} \label{sec:int-part} 

In this section we obtain a lower bound for the sequence $\left\{\ell(n)\right\}_{n\gs 1}$
by restricting our attention to zero-dimensional $\kk$-algebras of the form $\kk[x,y]/I$, where
$I$ is a homogeneous ideal of $\kk[x,y]$.  That is, we are concerned with $h$-vectors with $h_1=2$.  The main result is Theorem \ref{thm:lowbound} which shows that $\ell(n)$ is greater than or equal to the number of integer partitions of $n$ into distinct parts.  First, we develop the necessary background on integer partitions and \lex ideals.

For an $(x,y)$-primary monomial ideal $I$ in $\kk[x,y]$, we define 
\[
	\lam i = \lamb i I = |\{ x^{i-1}y^j \notin I \vl j \gs 0 \}|.
\]
Notice that $\lam i$ can also be viewed as a well-defined map from the set of monomial ideals in $\kk[x,y]$ to the positive integers. Further, from the definition of $\lambda_i$, we are able to write
\begin{equation}\label{eqn:ideal-gen}
	I = \left(y^{\lamb 1 I},x y^{\lamb 2 I},\dots,x^{i-1}y^{\lamb i I},\dots\right).
\end{equation}
As such, we have the following results detailing the natural correspondence between monomial ideals in two variables and integer partitions. Proposition \ref{prop:1-1-part} is known (see \cite{snellman04}) but we prove it here for the convenience of the reader.

\begin{lem}\label{lem:int-partition}
	Let $I$ be an $(x,y)$-primary monomial ideal in $R = \kk[x,y]$.  Then $(\lam 1, \lam 2, \cdots)$ is an integer partition of $\lambda(R/I)$.
\end{lem}
\begin{proof} 
Since $I$ is $(x,y)$-primary $\lambda(R/I)$ is finite.  We need to show $\lam 1 \gs \lam 2 \gs \cdots$ and $\lambda(R/I) = \sum \lam i$.  The first condition is true by the nature of monomial ideals.  The second condition holds because both $\lambda(R/I)$ and $\sum \lam i$ count the number of monomials not in $I$.
\end{proof}

\begin{prop}\label{prop:1-1-part}
		Let $R = \kk [x,y]$.  There exists a one-to-one correspondence between $(x,y)$-primary monomial ideals $I$ and integer partitions.	
\end{prop}
\begin{proof}
	Let $\mathfrak M$ be the set of $(x,y)$-primary monomial ideals in $R$ and $\mathfrak P$ be the set of integer partitions.  Define the map $\Phi:\mathfrak M \to \mathfrak P$ by 
\[
	\Phi(I) = (\lamb 1 I, \lamb 2 I, \cdots, \lamb s I),
\]
where $s$ is the largest integer such that $\lamb i I \neq 0$. 
	
	To show that $\Phi$ is one-to-one, let $I,J \in \mathfrak M$ such that $\Phi(I) = \Phi(J)$.  This forces $\lam i = \lamb i I = \lamb i J$ for all $i$.  In particular, for each $i$ we have that
\begin{align*}
	x^{i-1}y^0, x^{i-1}y^1,x^{i-1}y^2,\ldots, x^{i-1}y^{\lam i - 1} & \notin I,J; \\
	x^{i-1}y^{\lam i} & \in I,J.
\end{align*}
Hence $I = J$ as this completely defines the ideals.
	
To show $\Phi$ is onto, consider an integer partition $(\lst v t) \in \mathfrak P$ and let $I=\left(y^{v_1},x y^{v_2},\dots,x^{i-1}y^{v_t}\right)$. It is not hard to see that $\Phi(I) = (\lst v t)$.
\end{proof}

\begin{example}\label{part_ex}
Let $I=(y^3,x^2 y,x^3)\subseteq \kk[x,y]$.  The monomials not in $I$ are
\begin{align*}
& 1, y, y^2;\\
& x, xy, xy^2;\\
& x^2.
\end{align*}
Therefore the partition is $(3,3,1)$.  For those familiar with Ferrers diagrams,
the diagram is bottom and left justified diagram in the plane where the parts
are the number of boxes in the columns. 

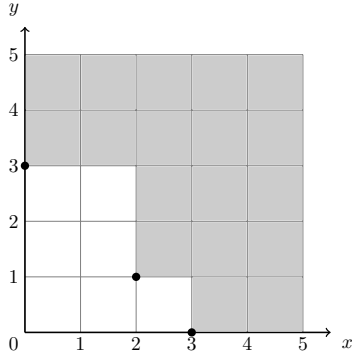
\begin{center}
\resizebox{6cm}{!}{%
	\begin{tikzpicture}[inner sep=4.7mm,
		rect/.style={rectangle,draw=black!20,fill=black!20,thick},														
		]
		
	   \path[use as bounding box] (-1,-0.5) rectangle (7,6.5);

		\draw[help lines] (0,0) grid (5,5);
		\draw[thick] (0,0) edge[->] (0,5.5);
		\draw[thick] (0,0) edge[->] (5.5,0);		
		
		\node at (-0.2,-0.2) [] {0};
		\node at (1,-0.2) [] {1};		
		\node at (2,-0.2) [] {2};		
		\node at (3,-0.2) [] {3};		
		\node at (4,-0.2) [] {4};		
		\node at (5,-0.2) [] {5};										
%		\node at (6,-0.2) [] {6};												
		\node at (5.8,-0.2) [] {$x$};												
		
		\node at (-0.2,1) [] {1};		
		\node at (-0.2,2) [] {2};		
		\node at (-0.2,3) [] {3};		
		\node at (-0.2,4) [] {4};		
		\node at (-0.2,5) [] {5};										
%		\node at (-0.2,6) [] {6};												
		\node at (-0.2,5.8) [] {$y$};

		% Shading
		\foreach \x in {0.5,...,4.5}
			\foreach \y in {3.5,4.5}
			{
				\node at (\x,\y) [rect] {};
			}
		
		\foreach \x in {2.5,...,4.5}
			\foreach \y in {1.5,2.5}
			{
				\node at (\x,\y) [rect] {};
			}		

		\foreach \x in {3.5,...,4.5}
			\foreach \y in {0.5}
			{
				\node at (\x,\y) [rect] {};
			}			

		% Generators
		\node at (3,0) [circle,fill=black,inner sep=1.5pt] {};
		\node at (2,1) [circle,fill=black,inner sep=1.5pt] {};		
%		\node at (1,4) [circle,fill=black,inner sep=1.5pt] {};		
		\node at (0,3) [circle,fill=black,inner sep=1.5pt] {};		
		
\end{tikzpicture}
}
\end{center}

\end{example}

 	The term \lex represents the \dfn{degree-lexicographical order} of monomials.  I.e., given a polynomial ring $R = \kk[\lst x n]$, 
\[
	x_1^{a_1}x_2^{a_2}\cdots x_n^{a_n} > x_1^{b_1}x_2^{b_2}\cdots x_n^{b_n}
\]
if the first non-zero entry of $(\sum (a_i - b_i), a_1-b_1, a_2-b_2, \ldots, a_n-b_n)$ is positive. This is a total order and as such, all monomials of degree $t$ are totally ordered.  A \dfn{\lex-segment} is the sequence of the first $s$ monomial terms in a given degree (in descending order).  We call an ideal generated by \lex-segments a \lex ideal. More precisely, we have the following definition.

\begin{defn}\label{lexideal_defn}
A monomial ideal $I\subseteq \kk[x_1,\dots,x_n]$ is is called a \lex ideal if 
for each $j\gs 0$, $I\cap \kk[x_1,\dots,x_n]_j$  is generated 
as a $\kk$-vector space by the first $\dim_\kk(I\cap  \kk[x_1,\dots,x_n]_j)$  
 monomials of degree $j$  in descending lexicographical order.
\end{defn}

\begin{rem}\label{rem:lex}
In the ring $R=\kk[x,y]$, an ideal $I\subseteq R$ is a \lex ideal if and only if it has the following property: if $x^iy^j\in I$ with $j\geq 1$, then $x^{i+1}y^{j-1}\in I$.
\end{rem}

	Given an $h$-vector, there are many ideals associated to it.  However, by Macaulay's Theorem, a mapping $h:\Z_{\gs 0} \to \Z_{\gs 0}$ is the Hilbert function for some graded ideal if and only if it is the Hilbert function of an ideal generated by \lex-segments. In other words, for each $h$-vector, there exists a unique ideal generated by \lex-segments. The next proposition relates \lex ideals to integer partitions

\begin{prop}\label{prop:lex-distinctparts}
Let $R=\kk[x,y]$ and $I$ be an $(x,y)$-primary monomial ideal.  Then the integer partition $(\lamb 1 I, \lamb 2 I, \cdots)$ has distinct entries if and only if $I$ is a \lex ideal.
\end{prop}
\begin{proof} 
	Assume that $(\lamb 1 I, \lamb 2 I,\dots)$ is a partition with distinct parts and write
\[
	I=\left(y^{\lamb 1 I},x y^{\lamb 2 I},\dots,x^{i-1}y^{\lamb i I},\dots\right)
\]	
as noted in \eqref{eqn:ideal-gen}.  We will use Remark \ref{rem:lex} to show that $I$ is \lex.  Suppose $x^a y^b\in I$; then there is a $\lamb i I$ such that $\ds x^{i-1}y^{\lamb i I}$ divides $x^a y^b$.  As such, $a+1\gs i$ and  $b-1\gs\lamb {i+1} I$ since $b\gs \lamb i I >\lamb{i+1} I$.  Therefore, $\ds x^i y^{\lamb {i+1} I}$ divides $x^{a+1} y^{b-1}$ and hence $x^{a+1}y^{b-1}$ is in $I$.

Conversely, let $I$ be a \lex ideal and  consider the partition $(\lamb 1 I,\lamb 2 I,\dots)$. By definition of $\lamb i I$, for each
$i\gs 1$ the monomial $\ds x^{i-1} y^{\lamb i I}$ is in $I$ but 
$\ds x^{i-1} y^{\lamb i I -1}\not\in I$.  Since $I$ is a \lex ideal,
we have $\ds x^{i} y^{\lamb i I -1}\in I$.  However, by definition of $\lamb {i+1} I$, 
$\ds x^{i} y^{\lamb {i+1} I -1}\not\in I$.  Therefore, $\ds \lamb i I>\lamb {i+1} I$
and all the parts are distinct.
\end{proof}

\begin{example}
	Consider partition $(5,4,2)$. The corresponding \lex ideal is $I = (y^5,xy^4,x^2y^2,x^3)$, with $h$-vector $(1,2,3,3,2)$, as can be seen by the lattice:

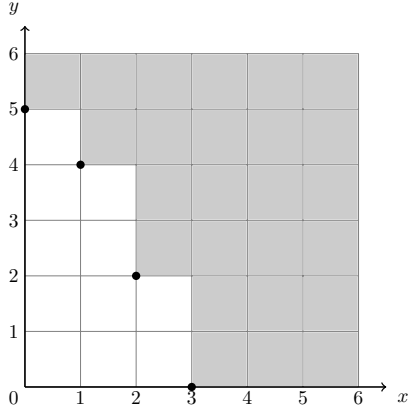
\begin{center}
\resizebox{6cm}{!}{%
	\begin{tikzpicture}[inner sep=4.7mm,
		rect/.style={rectangle,draw=black!20,fill=black!20,thick},														
		]
		
	   \path[use as bounding box] (-1,-0.5) rectangle (7,7.1);

		\draw[help lines] (0,0) grid (6,6);
		\draw[thick] (0,0) edge[->] (0,6.5);
		\draw[thick] (0,0) edge[->] (6.5,0);		
		
		\node at (-0.2,-0.2) [] {0};
		\node at (1,-0.2) [] {1};		
		\node at (2,-0.2) [] {2};		
		\node at (3,-0.2) [] {3};		
		\node at (4,-0.2) [] {4};		
		\node at (5,-0.2) [] {5};										
		\node at (6,-0.2) [] {6};												
		\node at (6.8,-0.2) [] {$x$};												
		
		\node at (-0.2,1) [] {1};		
		\node at (-0.2,2) [] {2};		
		\node at (-0.2,3) [] {3};		
		\node at (-0.2,4) [] {4};		
		\node at (-0.2,5) [] {5};										
		\node at (-0.2,6) [] {6};												
		\node at (-0.2,6.8) [] {$y$};

		% Shading
		\foreach \x in {0.5,...,5.5}
			\foreach \y in {5.5}
			{
				\node at (\x,\y) [rect] {};
			}
		
		\foreach \x in {1.5,...,5.5}
			\foreach \y in {4.5}
			{
				\node at (\x,\y) [rect] {};
			}		

		\foreach \x in {2.5,...,5.5}
			\foreach \y in {2.5,...,3.5}
			{
				\node at (\x,\y) [rect] {};
			}
			
		\foreach \x in {3.5,...,5.5}
			\foreach \y in {0.5,1.5}
			{
				\node at (\x,\y) [rect] {};
			}			

		% Generators
		\node at (3,0) [circle,fill=black,inner sep=1.5pt] {};
		\node at (2,2) [circle,fill=black,inner sep=1.5pt] {};		
		\node at (1,4) [circle,fill=black,inner sep=1.5pt] {};		
		\node at (0,5) [circle,fill=black,inner sep=1.5pt] {};		
		
\end{tikzpicture}
}
\end{center}

Notice that the $i^{\text{th}}$ entry of the partition $(5,4,2)$ counts the lattice points not in the ideal on the line $x = i-1$, for $i = 1,2,3$.  Likewise, the $j^{\text{th}}$ entry of the $h$-vector $(1,2,3,3,2)$ correspond to the number of lattice points not in the ideal on the line $y = -x + j$, for $0 \ls j \ls 4$. 
\end{example}

	Define the set $L_2(n) := \{ (1,\lst h s) \in L(n) \vl h_1 = 2 \}.$  Note that this set consists of all possible $h$-vectors of a zero-dimensional standard $\kk$-algebra of the form $\kk[x,y]/I$ (we are not allowing $I$ to contain $x$ or $y$). As such, we have the following.

\begin{thm}\label{thm:lowbound}
	The number of integer partitions of $n \in \Z_{\gs 1}$ into distinct parts is equal to $|L_2(n)|$.
\end{thm}
\begin{proof}
	 By Macaulay's Theorem \cite[Theorem 4.2.10]{BH}, each element of $L_2(n)$ corresponds uniquely to a \lex ideal.  Hence by Propositions \ref{prop:lex-distinctparts} and \ref{prop:1-1-part}, every element of $L_2(n)$ is in one-to-one correspondence with an integer partition with distinct parts.
\end{proof}

	Since $L_2(n) \subseteq L(n)$, we have a lower bound for the sequence $\left\{\ell(n)\right\}_{n\gs 1}$.  Although the bounds for $\left\{\ell(n)\right\}_{n\gs 1}$ are not tight, they have nice combinatorial interpretations.
Given this information, it is natural to ask the following questions: What are some better upper and lower bounds? 
Is it possible to write the sequence $\left\{\ell(n)\right\}_{n\gs 1}$ in a closed formula?

\section{More Properties}\label{sec:more-prop}

In this section we refine the set $L(n)$ of $h$-vectors of length $n$ in an attempt to obtain a closed form for $\ell(n)$. 
Let $L(n)$ be defined as in Section \ref{sec:fib-bound} and set
	\begin{align*}
		L_k(n) = &\{(1,h_1,h_2,\dots) \in L(n) \st h_1 = k \},
	\end{align*}
with $\ell_k(n) = |L_k(n)|$.  Notice that the $\ds L_k(n)$, $k\gs 1$ partition the set $L(n)$.

\begin{prop}\label{prop:tailseq}
  	Fix $n \in \Z_{\gs 0}$. If $k\gs 1$ such that $\ds \binom{k+2}{2} \gs n$, then $\ds \ell_{k+1}(n+1)=\ell_k(n)$.
\end{prop}

\begin{proof}

	Notice that if $(1,k,h_2,h_3,\dots) \in L_k(n)$, then $(1,k+1,h_2,h_3,\dots)$ satisfies Macaulay's condition (C). Let the map $\Psi: L_k(n) \to L_{k+1}(n+1)$ be defined by
	\[
		\Psi(1,k,h_2,h_3,\dots) = (1,k+1,h_2,h_3,\dots).
	\]
We claim that $\Psi$ is a bijection between $L_k(n)$ and $L_{k+1}(n+1)$ if  $\ds \binom{k+2}{2} \gs n$. As this map is certainly one-to-one for all $n\gs 1$, we only need to show it is onto.  Let $\bb h = (1,k+1,h_2,h_3,\dots)\in L_{k+1}(n+1)$.  Since the sum of the terms
	  equals $n+1$, we have 
	\[
		h_2\ls n+1 -1 -(k +1) =  n-k-1.  
	\]
By Condition (C), 
\[
	h_2\ls h_1^{\langle 1 \rangle}= (k+1)^{\langle 1 \rangle}=\binom{k+2}{2}.  
\]
However, $\ds \binom{k+2}{2} \gs n$ and thus
\[
 	h_2 \ls n-k-1\ls\binom{k+2}{2}-k-1 =\binom{k+1}{2} = k^{\langle 1 \rangle}.
\]
This shows that $(1,k,h_1,h_2,\dots)\in L_k(n)$ and hence $\Psi(1,k,h_1,h_2,\dots) = \bb h$.
\end{proof}

\begin{cor}
	For all $n,k \gs 1$, $\ell_k(n) \ls \ell_{k+1}(n+1)$.
\end{cor}
\begin{proof}
	Follows from the fact that $\Psi$ as defined in Proposition \ref{prop:tailseq} is injective.
\end{proof}

Finding a recurrence relation for the sequence $\left\{\ell(n)\right\}_{n\gs 1}$ appears to be difficult.  However, Proposition \ref{prop:tailseq} allows us to give a recursion formula for a sequence giving lower bound of $\ell(n)$.  Notice that $\{L_k(n)\}_{k\gs 1}$ form a partition of $L(n)$, and therefore $\ds\ell(n)=\sum_{k\gs 1}\ell_k(n)$. Given the quadratic nature of $\ds \binom{k+2}{2}$ versus the linear nature of $n$, we find that the recursion listed in Proposition \ref{prop:tailseq} represents the ``tail'' of the summation $\ds \sum_{k\gs 1}\ell_k(n)$. In particular, let $\ds s(n) = \min\left\{k \st n \ls \binom{k+2}{2} \right\}$ and define the sequence 
\[
	\tau(n) = \sum_{k = s(n)}^{n-1} \ell_k (n).
\]

\begin{cor} 
	The sequence $\tau(n)$ is a lower bound of $\ell(n)$ and is defined by the following recurrence relation:
	\begin{align*}
	\tau(1) &= 1; \\
	\tau(n) &= \begin{cases}
					\tau(n-1) + \ell_{s(n)}(n) & \text{if } s(n) = s(n-1)\\
					\tau(n-1) & \text{if } s(n) > s(n-1)
	\end{cases} \text{ for } n \gs 2.
	\end{align*}
\end{cor}
\begin{proof}
It is clear that $\tau(n)$ is a lower bound of $\ell(n)$. If $s(n) = s(n-1)$, then by Proposition \ref{prop:tailseq}, 
\[
	\tau(n-1) = \sum_{k = s(n-1)}^{n-2} \ell_k(n-1) = \sum_{k = s(n)}^{n-2} \ell_{k+1}(n) = \sum_{k = s(n) + 1}^{n-1} \ell_{k}(n).
\]
Hence $\tau(n) = \ell_{s(n)}(n) + \tau(n-1)$. If $s(n) > s(n-1)$, then $s(n) = s(n-1) +1$. Once again, by Proposition \ref{prop:tailseq}, 
\[
	\tau(n-1) = \sum_{k = s(n-1)}^{n-2} \ell_k(n-1) = \sum_{k = s(n-1)}^{n-2} \ell_{k+1}(n) = \sum_{k = s(n)}^{n-1} \ell_{k}(n).
\]
Therefore we have $\tau(n) = \tau(n-1)$.
\end{proof}

\section{Further Directions}

	As noted in Section \ref{sec:int-part}, $L_2(n)$ is a set whose cardinality represents the number of integer partitions of $n$ into distinct parts. This was obtained by restricting to elements of $L(n)$ whose first two entries are 1,2.  These are also the same $h$-vectors defined by 0-dimensional rings of the form $\kk[x,y]/I$ where $I$ is a graded ideal in $\kk[x,y]$.  In Section \ref{sec:more-prop}, this result was generalized with the sequences $|L_k(n)|$. Here the $L_k(n)$ are defined by $h$-vectors defined by 0-dimensional rings of the form $\kk[\lst x k]/I$.  This raises the following question: What algebraic conditions $\mathfrak C$ give rise to sequences with interesting counting properties? 

\begin{defn}\label{def:hseq}
 	The \dfn{$h$-sequence} of a condition $\mathfrak C$ is the sequence whose $n^{th}$ term is the number of $h$-vectors of length $n$ that satisfy $\mathfrak C$.  	
\end{defn}

  One of the fundamental properties a 0-dimensional $\kk$-algebra could have is the weak Lefschetz property (WLP).  This property is geometric in origin, and is a current topic of study in algebra and combinatorics.  As shown in \cite[Proposition 3.5]{harima03}, given an integer vector $\bb h = (1, \lst h s)$, $\bb h$ is the $h$-vector of a graded 0-dimensional $\kk$-algebra having the WLP if and only if $\bb h$ is a unimodal $h$-vector such that the positive part of the first difference is also an $h$-vector. Thus if we let $\mathfrak C$ be the WLP, we are able to compute the $h$-sequence of $\mathfrak C$.  We list this sequence in Figure \ref{fig:h-seq} along with some other interesting conditions.  Apart from the sequence $L_2(n)$, none of these sequences are found on the on-line encyclopedia of integer sequences \cite{oeis}.

\begin{figure}[h!]
	\begin{tabular}{r|ccccccccc}
		$\mathfrak C$ 	&	$\ell(n)$	&	 WLP 	&	 Unimodal 	&	 Symmetric 	&	 $L_2(n)$ 	&	 $L_3(n)$ 	&	 $L_4(n)$ 	&	 $L_5(n)$ 	\\
				\hline																	
				\hline																	
				1 	&	1	&	1	&	 1 	&	 1 	&	0	&	0	&	0	&	0	\\
				2 	&	1	&	1	&	 1 	&	 1 	&	0	&	0	&	0	&	0	\\
				3 	&	2	&	2	&	 2 	&	 1 	&	1	&	0	&	0	&	0	\\
				4 	&	3	&	3	&	 3 	&	 2 	&	1	&	1	&	0	&	0	\\
				5 	&	5	&	5	&	 5 	&	 2 	&	2	&	1	&	1	&	0	\\
				6 	&	8	&	8	&	 8 	&	 3 	&	3	&	2	&	1	&	1	\\
				7 	&	12	&	12	&	 12 	&	 2 	&	4	&	3	&	2	&	1	\\
				8 	&	18	&	18	&	 18 	&	 4 	&	5	&	5	&	3	&	2	\\
				9 	&	27	&	27	&	 27 	&	 3 	&	7	&	7	&	5	&	3	\\
				10 	&	40	&	40	&	 40 	&	 4 	&	9	&	11	&	7	&	5	\\
				11 	&	57	&	56	&	 56 	&	 3 	&	11	&	15	&	11	&	7	\\
				12 	&	82	&	80	&	 80 	&	 6 	&	14	&	21	&	16	&	11	\\
				13 	&	116	&	112	&	 112 	&	 4 	&	17	&	29	&	23	&	16	\\
				14 	&	163	&	155	&	 155 	&	 7 	&	21	&	39	&	33	&	23	\\
				15 	&	227	&	213	&	 213 	&	 4 	&	26	&	52	&	46	&	33	\\
				16 	&	313	&	290	&	 290 	&	 8 	&	31	&	70	&	63	&	46	\\
				17 	&	428	&	389	&	 390 	&	 5 	&	37	&	91	&	87	&	64	\\
				18 	&	583	&	522	&	 523 	&	 10 	&	45	&	119	&	117	&	89	\\
				19 	&	788	&	694	&	 696 	&	 5 	&	53	&	155	&	157	&	121	\\
				20 	&	1059	&	915	&	 920 	&	 13 	&	63	&	199	&	210	&	164	
	\end{tabular}
	\caption{$h$-sequences of various conditions}\label{fig:h-seq}
\end{figure}

\section{Acknowledgements}

We would like to thank Craig Huneke for the initial motivation for the problem.  Additionally, the calculations in this note were inspired by many Macaulay2 \cite{M2} computations.  The interested reader should contact the authors if they would like Macaulay2 code for investigating these types of objects further.

%\input{bibliography}
%\bibliography{ourBib}{}

\begin{bibdiv}
\begin{biblist}

\bib{BH}{book}{
      author={Bruns, Winfried},
      author={Herzog, J{\"u}rgen},
       title={Cohen-{M}acaulay rings},
      series={Cambridge Studies in Advanced Mathematics},
   publisher={Cambridge University Press},
     address={Cambridge},
        date={1993},
      volume={39},
        ISBN={0-521-41068-1},
      review={\MR{1251956 (95h:13020)}},
}

\bib{manoj}{article}{
      author={Caviglia, Giulio},
      author={Kummini, Manoj},
       title={Poset embeddings of {H}ilbert functions},
        date={2012},
     journal={ar{X}iv:1009.4488},
}

\bib{CL}{article}{
      author={Clements, G.~F.},
      author={Lindstr{\"o}m, B.},
       title={A generalization of a combinatorial theorem of {M}acaulay},
        date={1969},
     journal={J. Combinatorial Theory},
      volume={7},
       pages={230\ndash 238},
      review={\MR{0246781 (40 \#50)}},
}

\bib{FR}{incollection}{
      author={Francisco, Christopher~A.},
      author={Richert, Benjamin~P.},
       title={Lex-plus-powers ideals},
        date={2007},
   booktitle={Syzygies and {H}ilbert functions},
      series={Lect. Notes Pure Appl. Math.},
      volume={254},
   publisher={Chapman \& Hall/CRC, Boca Raton, FL},
       pages={113\ndash 144},
         url={http://dx.doi.org/10.1201/9781420050912.ch4},
      review={\MR{2309928 (2008a:13015)}},
}

\bib{M2}{misc}{
      author={Grayson, Daniel~R.},
       title={Macaulay2, a software system for research in algebraic geometry},
     address={Available at \url{http://www.math.uiuc.edu/Macaulay2/}},
}

\bib{harima03}{article}{
      author={Harima, Tadahito},
      author={Migliore, Juan~C.},
      author={Nagel, Uwe},
      author={Watanabe, Junzo},
       title={The weak and strong {L}efschetz properties for {A}rtinian
  {$K$}-algebras},
        date={2003},
        ISSN={0021-8693},
     journal={J. Algebra},
      volume={262},
      number={1},
       pages={99\ndash 126},
         url={http://dx.doi.org/10.1016/S0021-8693(03)00038-3},
      review={\MR{1970804 (2004b:13001)}},
}

\bib{linusson99}{article}{
      author={Linusson, Svante},
       title={The number of {$M$}-sequences and {$f$}-vectors},
        date={1999},
        ISSN={0209-9683},
     journal={Combinatorica},
      volume={19},
      number={2},
       pages={255\ndash 266},
         url={http://dx.doi.org/10.1007/s004930050055},
      review={\MR{1723043 (2000k:05012)}},
}

\bib{macaulay27}{article}{
      author={MacAulay, F.~S.},
       title={Some {P}roperties of {E}numeration in the {T}heory of {M}odular
  {S}ystems},
        date={1927},
        ISSN={0024-6115},
     journal={Proc. London Math. Soc.},
      volume={S2-26},
      number={1},
       pages={531},
         url={http://dx.doi.org/10.1112/plms/s2-26.1.531},
      review={\MR{1576950}},
}

\bib{MM}{article}{
      author={Mermin, Jeff},
      author={Murai, Satoshi},
       title={Betti numbers of lex ideals over some {M}acaulay-{L}ex rings},
        date={2010},
        ISSN={0925-9899},
     journal={J. Algebraic Combin.},
      volume={31},
      number={2},
       pages={299\ndash 318},
         url={http://dx.doi.org/10.1007/s10801-009-0192-1},
      review={\MR{2592080 (2011b:13066)}},
}

\bib{oeis}{article}{
       title={The on-line encyclopedia of integer sequences},
        date={2013},
     journal={Published electronically},
        note={http://oeis.org},
}

\bib{PS}{article}{
      author={Peeva, Irena},
      author={Stillman, Mike},
       title={Open problems on syzygies and {H}ilbert functions},
        date={2009},
        ISSN={1939-0807},
     journal={J. Commut. Algebra},
      volume={1},
      number={1},
       pages={159\ndash 195},
         url={http://dx.doi.org/10.1216/JCA-2009-1-1-159},
      review={\MR{2462384 (2009i:13024)}},
}

\bib{snellman04}{article}{
      author={Snellman, Jan},
      author={Paulsen, Michael},
       title={Enumeration of concave integer partitions},
        date={2004},
        ISSN={1530-7638},
     journal={J. Integer Seq.},
      volume={7},
      number={1},
       pages={Article 04.1.3, 10},
      review={\MR{2049698}},
}

\bib{gwyn}{misc}{
      author={Whieldon, Gwyneth~R.},
       title={Infinite free resolutions over monomial rings in two variables},
        date={2013},
      volume={math.AC},
        note={arXiv:1308.0179},
}

\end{biblist}
\end{bibdiv}
%\bibliographystyle{plain}
%\end{document}


\end{document}